\documentclass[12pt]{amsart}
\usepackage{graphicx,epsfig}
\usepackage{amssymb,amsfonts,amsmath,amsthm,dsfont,wasysym,pifont,stmaryrd}
\usepackage{epstopdf,yfonts}
\usepackage{hyperref}
\newtheorem{theorem}{Theorem}[section]
\newtheorem*{theorem'}{Theorem }

\newtheorem{prop}[theorem]{Proposition}

\newtheorem{cor}[theorem]{Corollary}

\newtheorem{lemma}[theorem]{Lemma}
\newtheorem{definition}[theorem]{Definition}

\renewcommand{\~}[1]{\overline{#1}}
\renewcommand{\geq}{\geqslant}
\renewcommand{\leq}{\leqslant}

\newcommand{\<}{\left\langle}
\renewcommand{\>}{\right\rangle}
\newcommand{\8}{\infty}

\renewcommand{\:}{\colon}
\renewcommand{\a}{\alpha}

\renewcommand{\b}{\beta}

\newcommand{\Bi}{B^1}

\renewcommand{\Cup}[2]{\underset{#1}{\overset{#2}{\cup} }}

\newcommand{\e}{\epsilon}
\newcommand{\EL}{\textnormal{EL}}

\newcommand{\F}{\Phi}

\newcommand{\g}{\gamma}
\newcommand{\G}{\Gamma}

\renewcommand{\H}{\mathcal{H}}

\newcommand{\Inf}[1]{\underset{#1}{\inf}}
\renewcommand{\int}{\varint}
\newcommand{\Isom}{\mathrm{Isom}}

\renewcommand{\L}{\Lambda}

\renewcommand{\max}[1]{\underset{#1}{\mathrm{max}}}

\newcommand{\norm}{\trianglelefteqslant}

\newcommand{\om}{\omega}

\newcommand{\PSL}{\textnormal{PSL}}

\newcommand{\R}{\mathcal{R}}
\renewcommand{\Re}{\mathbb{R}}

\newcommand{\SL}{\textnormal{SL}}

\newcommand{\Sum}[2]{\underset{#1}{\overset{#2}{\sum} }}

\newcommand{\Un}{\mbox{$\mathcal{U}$}}

\newcommand{\Z}{\mathbb{Z}}
\newcommand{\Zi}{Z^1}

\title{Relative Property (T) and the Vanishing of the first $\ell^2$-Betti number}
\author{Talia Fern\'os}
\address{Hebrew University, Jerusalem Israel 
}
\email{fernos@math.huji.ac.il}
\begin{document}
\maketitle


\begin{abstract}
In this paper, we show that certain families with relative property (T) have trivial first $\ell^2$-Betti number. We apply this to the elementary matrix group $\EL_n(\R)$ where $\R$ is any countable unital ring of characteristic 0.
\end{abstract}

\section{Introduction} 

The notion of $\ell^2$-Betti numbers has been studied for a variety of objects from groups to measurable folliations \cite{Atiyah}, \cite{CG}, \cite{Connes}. They were studied by D. Gaboriau who showed that if the (probability measure preserving and free) actions of two groups are orbit equivalent then their $\ell^2$-Betti numbers are the same \cite{Gab02}. Namely, $\ell^2$-Betti numbers are an invariant of orbit equivalence.

A generalization of the usual Betti numbers, $\ell^2$-Betti numbers count the ``module-dimension" of certain infinite dimensional cohomology groups.  See \cite{Gaboriau} for a brief introduction to $\ell^2$-Betti numbers for groups and equivalence relations.

Our objective here is to prove the following:

\begin{theorem}\label{Th RelT ell2 Betti}
Let $(\G, A)$ be a group pair with relative property (T). If $A$ is infinite  and weakly normal or weakly quasi-normal in $\G$ then $\b_1(\G) = 0$.
\end{theorem}

{\bf{Remark:}} In the case when $A$ is normal in $\G$ the above theorem follows from Proposition 3.5.1 in Florian Martin's thesis \cite{Martin}. The proposition reads: if $G$ is locally compact and $N$ is a closed normal subgroup such that $(G,N)$ has relative (T), then $H^1(G,\pi)=0$ for every representation $\pi$ without $N$-invariant vectors.

Recall the following two definitions:

\begin{definition}
A subgroup $A \leq \G$ is said to be weakly normal if there exists a well ordered (by inclusion) family of intermediate subgroups $\{N_t : 0 \leq t \leq \tau\}$ of $\G$ with $N_0 = A$ and $N_\tau= \G$ and such that for each $0 < t \leq \tau$ the following inclusion is normal:
\begin{equation}
\nonumber
\Cup{s < t}{}N_s \norm N_t
\end{equation}
\end{definition}

\begin{definition}[Popa]
A subgroup $A \leq \G$ is said to be weakly quasi-normal if there exists a well ordered (by inclusion) family of intermediate subgroups $\{N_t : 0 \leq t \leq \tau\}$ of $\G$ with $N_0 = A$ and $N_\tau= \G$ and 
\begin{itemize}
\item for each non-limit ordinal $0 < t \leq \tau$ the subgroup $N_t = \< \g \: |N_{t-1} \cap \g N_{t-1} \g^{-1}| =\8 \>$;
\item for each limit ordinal $0 < t \leq \tau$ the subgroup $N_t =\Cup{s < t}{}N_s $.
\end{itemize}
\end{definition}

Since many of the family of examples of nontrivial group pairs with relative property (T) are of the form $(\G,A)$ where $A$ is normal and amenable \cite{Valette}, \cite{Fernos2006}, Theorem \ref{Th RelT ell2 Betti} can be seen as a natural extension of the following:

\begin{theorem}\label{Th ell2 Betti numbers amenable and T}{\cite[Corollary 6, Theorem D]{BV}}
If $\G$ is a group satisfying either of the following two conditions then $\b_1(\G)= 0$:
\begin{enumerate}
  \item $\G$ has property (T);
  \item $\G$ is finitely generated and contains a normal amenable subgroup.
\end{enumerate}
\end{theorem}

Examples of groups having $\b_1 > 0$ are lattices (both uniform and nonuniform) in $\PSL_2(\Re)$ \cite{BV}{Example p. 318}. For a large class of examples see \cite{PetThom}.  

I would like to thank Alain Valette for reading this and making useful suggestions. This work was part of my doctoral thesis.

\section{Some preliminaries}

The celebrated theorem of Delorme and Guichardet says that property (T) is equivalent to property FH. The same is true for relative property (T) and it is through this fact that we will prove our theorem.

 \begin{definition}
A group pair $(\G, A)$ has relative property FH if every representation $\a \: \G \to \Isom(\H)$ on a Hilbert space $\H$, has an $A$-fixed point $v \in \H$.
\end{definition}

We will be thinking of the affine Hilbert spaces as vector spaces with an affine structure. By this we mean that we choose once and for all, an origin $0 \in \H$. This is equivalent to choosing an isomorphism $\Isom(\H) \cong \Un(\H) \ltimes\H$. 
 
\begin{theorem}[\cite{Jolissaint05},\cite{Cornullier},\cite{FernosThesis}]\label{relTrelFH} A group pair $(\G,A)$ has relative property (T) if and only if it has relative property FH.
\end{theorem}

 
\subsection{Affine Actions and Relative Cohomology}
 
Let $\H$ be a Hilbert space and $\pi \: \G \to \Un(\H)$ be a unitary representation of $\G$. We will need to recall some definitions: 

\begin{definition}

\begin{enumerate}
  \item The space of $\pi$-cocycles is:
\begin{equation}
\Zi(\G, \pi) := \{b \: \G \to \H | b(\g_1\g_2) = b(\g_1) + \pi(\g_1)b(\g_2)\} \nonumber
\end{equation}
  \item The space of $\pi$-relative-coboundaries is:
\begin{equation}
\Bi(\G, A, \pi) := \{b \in \Zi | b(a) = \pi(a) v - v \text{ for some } v \in \H \text{ and for all } a \in A\}.
 \nonumber
\end{equation}
  \item The first cohomology is the quotient:
\begin{equation}
H^1(\G,\pi) : = \Zi(\G, \pi)/ \Bi(\G,\G, \pi).
\nonumber
\end{equation}
\item The relative first cohomology is the quotient:
\begin{equation}
H^1(\G,A, \pi) : = \Zi(\G, \pi)/ \Bi(\G,A, \pi).
\nonumber
\end{equation}
\end{enumerate}
\end{definition}

Observe that to each cocycle $b \in \Zi(\G,\pi)$ there corresponds an isometric action 
\begin{equation}
\nonumber
\a_b(\g)(v) := \pi(\g)v + b(\g).
\end{equation}

The cocycle identity assures that this is an action. We will refer to this action as the one induced by $b$ and will often drop the subscript if there is no ambiguity. Conversely, having fixed an origin $0 \in \H$, any isometric action $\a \: \G \to \Isom(\H)$ gives rise to the map $b_\a(\g) := \a(\g)(0)$ which is guaranteed to be a cocycle by the fact that the action is affine.

The first cohomology $H^1(\G, \pi)$ is developed in order to understand what kinds of affine $\G$-actions can arise with linear part $\pi$, modulo the trivial ones (i.e. those arising out of the coboundaries in $\Bi(\G,\G, \pi)$). Analogously, one may understand $H^1(\G,A, \pi)$ as the possible ways of assembling an isometric action of $\G$, with linear part $\pi$ which are trivial relative to $A$. It is an easy exercise to show that a cocycle is a coboundary on $A$ exactly when it is bounded on $A$.

We can now rephrase Theorem \ref{relTrelFH} in terms of relative cohomology:

\begin{theorem}
A group pair $(\G, A)$ has relative property (T) if and only if $H^1(\G,A, \pi)=0$ for every unitary $\G$-representation $\pi$. 
\end{theorem}


\section{Proof of Theorem \ref{Th RelT ell2 Betti}}

Theorem \ref{Th RelT ell2 Betti}  follows from three results\footnote{It is well known that for $(\G,A)$ to have relative property (T), with $A$ infinite, it is necessary for $\G$ to be nonamenable. See for example \cite{Fernos2006}.}:

\begin{theorem}[\cite{BV}{Corollary 4}]
Let $\lambda_\G \: \G \to \Un(\ell^2(\G))$ be the left regular representation of $\G$. If $\G$ is nonamenable and $H^1(\G, \lambda_\G)=0$ then $\b_1(\G) = 0$. 
\end{theorem}

\begin{prop}\label{wn}
Let $(\G,A)$ be a group-subgroup pair, and $A$ infinite and weakly normal  in $\G$. If $\pi \: \G \to \Un(\H)$ is a unitary representation with no nontrivial $A$-invariant vectors and $H^1(\G,A, \pi)=0$ then $ H^1(\G, \pi) = 0$.
\end{prop}

\begin{prop}\label{wqn}
Let $(\G,A)$ be a group-subgroup pair, and $A$ infinite and weakly quasi-normal in $\G$. If $\pi \: \G \to \Un(\H)$ is a $C_0$-unitary representation and $H^1(\G,A, \pi)=0$ then $ H^1(\G, \pi) = 0$.
\end{prop}

We now prove these propositions, first in the case where $A$ is weakly normal and then when it is weakly quasi-normal in $\G$.

\subsection{$A$ weakly normal}

Under  the same assumptions as Proposition \ref{wn}, we prove the following two lemmas, the last of which proves the proposition.

\begin{lemma}\label{cocind}
If $A \leq N \leq \G$ and $b \in \Bi(\G, N,\pi)$ then $b \in \Bi(\G, N',\pi)$ for any subgroup $N'$ of the normalizer of $N$ in $\G$.
\end{lemma}

\begin{proof}
Suppose that $b \in \Bi(\G, N,\pi)$ and $v \in \H$ such that $b(a) = \pi(a)v-v$ for every $a \in N$. Now, fix $\g \in N'$ and we have:
 \begin{eqnarray}
b(\g^{-1} a \g) &= & b(\g^{-1} a) + \pi(\g^{-1} a) b(\g)  \nonumber \\
 & = & b(\g^{-1}) + \pi(\g^{-1})b(a) +\pi(\g^{-1} a)b(\g) \nonumber \\
= \pi(\g^{-1} a \g)v - v & = & b(\g^{-1}) + \pi(\g^{-1})( \pi(a) v - v) + \pi(\g^{-1} a)b(\g) \nonumber
\end{eqnarray}

Recalling the fact that $b(\g^{-1}) = -\pi(\g^{-1}) b(\g)$ is a consequence of the cocycle identity, we multiply the last line above by $\pi(\g)$ and get:

\begin{eqnarray}
\pi(a)(v+ b(\g) - \pi(\g)v) & = & v+ b(\g) - \pi(\g)v. \nonumber
\end{eqnarray}

Since $\pi$ has no nontrivial $A$-invariant vectors, and the last equation holds for every $a\in N \geq A$ it follows that $v+ b(\g) - \pi(\g)v=0$, that is to say $ b(\g) = \pi(\g)v -v$ and therefore, $b \in \Bi(\G, N',\pi)$.
\end{proof}

\begin{lemma}
For every $b \in \Zi(\G, \pi)$ either $b \in \Bi(\G, \G, \pi)$ or $b \notin \Bi(\G, A, \pi)$.
\end{lemma}

\begin{proof}
Suppose that $b \notin \Bi(\G, \G, \pi)$. Then the set $\{t\: b\notin \Bi(\G, N_t, \pi)\}$ contains $\tau$ (recall that $N_\tau = \G$) and therefore the set is nonempty and must have a minimal element $t_0$. If $t_0= 0$ then we are done. Otherwise, $b \in  \Bi(\G, \Cup{s<t_0}{}N_s, \pi)$ but $b \notin \Bi(\G, N_{t_0}, \pi)$ which contradicts the previous lemma since $ \Cup{s<t_0}{}N_s \norm N_{t_0}$.
\end{proof}

\subsection{wq-Normality}

Under  the same assumptions, the following proves Proposition \ref{wqn}.

\begin{lemma}
Let $A \leq \G$ be a wq-normal subgroup, with $A$ infinite and $\pi \: \G \to \Un(\H)$ a  $C_0$-representation. Then for every $b \in \Zi(\G, \pi)$ either $b \in \Bi(\G, \G, \pi)$ or $b \notin \Bi(\G, A, \pi)$.
\end{lemma}

\begin{proof}
Let $\{N_t: 0\leq t \leq \tau\}$ be the well ordered collection of subgroups guaranteed by the definition of  weak quasi-normality for $A\leq \G$.

Let $b\notin  \Bi(\G, \G, \pi)$. We aim to show that $b\notin  \Bi(\G, A, \pi)$. Consider the set $\{t\: b\notin \Bi(\G, N_t, \pi)\}$. As this set contains $\tau$ it is not empty and hence contains a minimal element $t_0$. If $t_0=0$, then we are done. By contradiction suppose that $t_0>0$. 

We claim that $t_0$ is not a limit ordinal. If $t_0$ is a limit ordinal,  then $N_{t_0} = \Cup{t<t_0}{} N_t$. Let $v_t\in \H$ be the vector making  $b|_{N_t} (\g) = \pi(\g)v_t - v_t$ a relative coboundary on each $N_t$. Then $v_t - v_0$ is an invariant vector for the infinite subgroup $N_0= A$ and since the representation $\pi$ is $C_0$  it must be that $v_t = v_0$, which shows that the restriction of $b$ to $N_{t_0}$ is a coboundary, contradicting the minimality of $t_0$.  

Therefore $t_0$ is not a limit ordinal and by minimality, the restriction of $b$ to $N_{t_0-1}$  must be a coboundary. 

Recall that $N_{t_0}$ is generated by a set $S_{t_0}$ where every $h \in S_{t_0}$ has the property that $$|N_{t_0-1}\cap hN_{t_0-1}h^{-1}| = \8.$$ We now show that there is a $v\in \H$ such that $b(h) = \pi(h)v-v$ for each $h \in S_{t_0}$. The cocycle relation then implies that $b \in \Bi(\G, N_{t_0}, \pi)$ contradicting the minimality of $t_0$.

Let $v\in \H$ such that $b|_{ N_{t_0-1}}(\g) = \pi(\g)v-v$. For $h \in S_{t_0}$ consider $\~\g \in h N_{t_0-1} h^{-1} \cap N_{t_0-1}$. Then, $\~\g = h\g h^{-1}$ for some $\g \in N_{t_0-1}$ and so

\begin{eqnarray*}
\pi(\~\g)v - v\enspace=\enspace b(\~\g)& =& b(h\g h^{-1}) \\
&=& b(h) + \pi(h) (\pi(\g)v-v) - \pi(h)(\pi(\g)\pi(h^{-1})b(h)) \\
&=& b(h) + \pi(\~\g)\pi(h)v - \pi(h) v - \pi(\~\g)b(h).
\end{eqnarray*}

Rearranging terms we have that

\begin{eqnarray*}
\pi(\~\g)(\pi(h)v - v - b(h))& =& \pi(h)v - v - b(h). 
\end{eqnarray*}

Of course this holds for infinitely many $\~\g \in h N_{t_0-1} h^{-1} \cap N_{t_0-1}$ and since $\pi$ is $C_0$ we conclude that $b(h) = \pi(h)v -v$.

\end{proof}

\section{Application to $\EL_n(\R)$ for $n \geq 3$}

Let $\R$ be a countable (not necessarily commutative) ring with unit and characteristic 0. Consider the group $\EL_n(\R)$ generated by elementary matrices and $n\geq 3$. We have many options here for a weakly quasi-normal subgroup. For example, since $1\in \R$ it follows that $\EL_n(\R) \geq \SL_n(\Z)$. Instead, we choose $A = \< E_{i,n}(1): i = 1, \dots, n-1\> \cong \Z^{n-1}$ to be the subgroup generated by the elementary matrices corresponding to the last column. Of course, $A$ is contained in a copy of $\SL_{n-1}(\Z) \ltimes A$ inside $\EL_n(\R)$ and therefore,  $(\EL_n(\R), A)$ has relative property (T). 

Consider  $$N_1 = \< E_{i,j}(r): r\in \R, i = 1, \dots, n-1, j= 1, \dots, n, \text{ and }  i\neq j\> \cong \EL_{n-1}(\R)\ltimes \R^{n-1}$$ and $$N_2 = \< E_{i,j}(r): r\in \R, i = 1, \dots, n, j= 1, \dots, n,  \text{ and }  i\neq j\> \cong \EL_n(\R).$$

For each $i_0$, we see that the infinite cyclic subgroup $A_{i_0} = \< E_{i_0,n}(1)\>$ is centralized by $$C_{i_0} := \{E_{i_0,j}(r): r\in \R,  j=1, \dots n-1, \text{ and } j \neq i_0\} \cup \{E_{i,n}(r): r\in \R \text{ and } i\neq i_0 \}.$$

This means that $\g A \g^{-1} \cap A \supset A_{i_0}$ for each $\g \in C_{i_0}$ and hence these intersections are infinite. 

Now, as $N_2$ is generated by $N_1$ together with $\{E_{n,j}(r): r\in \R  \text{ and }   j=1, \dots n-1\}$ we need only see that $E_{n,j}(\cdot)$ centralizes some infinite subgroup of $N_1$ for each $j \neq n$. But of course, as before, $E_{n,j}(\cdot)$ commutes with $E_{i,j}(\cdot)$ for any $i \neq n, j$.

\begin{cor}
For $n\geq 3$ and $\R$ a unital ring of characteristic 0, the first $\ell^2$-Betti number $\b_1(\EL_n(\R))=0$. 
\end{cor}


\bibliography{bib}
\bibliographystyle{alpha}

\end{document}